\documentclass[psamsfonts]{amsart}
\usepackage[margin=1.5in]{geometry}

\usepackage{amssymb,amsfonts,amsthm}
\usepackage[all,arc]{xy}
\usepackage{enumerate}
\usepackage{mathrsfs}
\usepackage{tabto}
\usepackage{tikz}
\usepackage{tikz-cd}
\usepackage{hyperref}
\usepackage{verbatim}
\usepackage{comment}
\usepackage{appendix}
\usepackage{graphicx}
\usepackage{soul}

\newtheorem{thm}{Theorem}[section]

\newtheorem{prop}[thm]{Proposition}
\newtheorem{lem}[thm]{Lemma}

\newtheorem{quest}[thm]{Question}
\newtheorem{claim}[thm]{Claim}

\theoremstyle{definition}
\newtheorem{defn}[thm]{Definition}

\newtheorem{remark}[thm]{Remark}

\newcommand{\N}{\mathbb N}

\newcommand{\Dom}{\text{Dom}}
\newcommand{\asi}{\text{asi}}

\newcommand{\asdim}{\text{asdim}}

\theoremstyle{remark}

\makeatletter
\let\c@equation\c@thm
\makeatother
\numberwithin{equation}{section}

\newcommand{\rest}{\restriction}
\newcommand\mbf{\mathbf}
\newcommand\mcal{\mathcal}
\newcommand{\LOCAL}{\textsf{LOCAL }}

\begin{document}

\author{Jan Grebík and Cecelia Higgins}

\address{JG: Faculty of Informatics, Masaryk University, Botanicka 68A, 60200 Brno, Czech Republic}

\address{CH: Department of Mathematics, University of California, Los Angeles, 520 Portola Plaza, Los Angeles, CA 90095, United States}
    
\title{Complexity of Finite Borel Asymptotic Dimension}
\date{\today}

\begin{abstract}
    We show that the set of locally finite Borel graphs with \emph{finite Borel asymptotic dimension} is $\mbf{\Sigma}^1_2$-complete.
    The result is based on a combinatorial characterization of finite Borel asymptotic dimension for graphs generated by a single Borel function.
    As an application of this characterization, we classify the complexities of digraph homomorphism problems for this class of graphs.
\end{abstract}

\maketitle

\begin{section}{Introduction}

The study of \emph{Borel graphs} and \emph{Borel equivalence relations} has figured prominently in descriptive set theory over the past several decades (see, e.g., \cite{kst1999,km2004,jackson2019,mathe2019,km2020,pikhurko2021,bernshteyn2022,gv2024,kechris2024}).
One of the central themes in this area is the notion of \emph{Borel hyperfiniteness}.
A Borel graph $\mcal{G}$ is \emph{Borel hyperfinite} if it can be written as a countable increasing union of Borel graphs with finite connected components, i.e., $\mcal{G} = \bigcup_{n \in \N} \mcal{G}_n$, where for each $n \in \N$, $\mcal{G}_n$ is a Borel graph with finite connected components and $\mathcal{G}_n\subseteq\mathcal{G}_{n+1}$. 

Hyperfinite structures arise in various areas of mathematics and theoretical computer science as non-trivial structures of minimal complexity (see, e.g., \cite{connes1976,cfw1981,hkl1990,km2004,matui2012,lovasz2012,elek2018}).
In particular, in the theory of countable Borel equivalence relations, the class of Borel hyperfinite equivalence relations is the minimum non-trivial element in the Borel reducibility hierarchy by the Glimm--Effros dichotomy \cite{glimm1961,effros1965,effros1980,hkl1990}.
However, many basic problems about Borel hyperfiniteness remain open (see \cite{kechris2024} for a recent survey).
Among these problems, one of the deepest centers around the \emph{complexity} of hyperfiniteness: While it is immediate from the definition that hyperfiniteness is a $\mbf{\Sigma}^1_2$ notion, it is unknown whether the set of codes of Borel hyperfinite graphs is \emph{$\mbf{\Sigma}^1_2$-complete}, that is, maximally complex \cite{djk1994} (see Section~\ref{sec:Prelim} for the formal definitions of the complexity notions).
The $\mathbf{\Sigma}^1_2$-completeness of Borel hyperfiniteness would rule out a number of conjectured dichotomy theorems \cite{kechris2024}.

In \cite{cjmst2023}, the authors introduce the notion of \emph{Borel asymptotic dimension} as an extension of Gromov's classical notion of asymptotic dimension for metric spaces \cite{gromov1993} to the definable context; they also show that if a locally finite Borel graph has finite Borel asymptotic dimension, then it is Borel hyperfinite \cite[Theorem~7.1]{cjmst2023}.
Despite its recency, Borel asymptotic dimension has become an invaluable tool in the study of Borel hyperfiniteness (see, e.g., \cite{by2023,nv2023,is2024}).

While determining the exact complexity of Borel hyperfiniteness remains a major open problem, our main result shows that the notion of finite Borel asymptotic dimension is $\mbf{\Sigma}^1_2$-complete.

	\begin{thm}
    \label{thm:main}
		The set of locally finite Borel graphs having finite Borel asymptotic dimension is $\mbf{\Sigma}^1_2$-complete.
	\end{thm}

Theorem~\ref{thm:main} builds on the celebrated paper of Todorčević and Vidnyánszky \cite{tv2021}, who showed that locally finite Borel graphs having finite Borel chromatic number form a $\mbf{\Sigma}^1_2$-complete set, and adds to the growing body of complexity results in Borel combinatorics (see, e.g., \cite{bcggrv2021,thornton2022,fsv2024}).

\medskip

The main tool we use to prove Theorem~\ref{thm:main} involves transforming the geometric question of whether a graph generated by a single Borel function has finite Borel asymptotic dimension (see Section~\ref{sec:FinAsDim} for the definition) into a purely combinatorial question concerning the existence of certain forward-recurrent sets.
A Borel function $f$ on a standard Borel space $X$ is \emph{acyclic} if $f^k(x) \neq x$ for every $x \in X$ and $k \in \N^+ = \N \setminus \{0 \}$.
We write $\mcal{G}_f$ for the Borel graph generated by $f$, that is, the graph on $X$ where vertices $x, y \in X$ are adjacent if and only if $x \neq y$ and either $f(x) = y$ or $f(y) = x$.
A set $H \subseteq X$ is {\it hitting for $f$} if, for each $x \in X$, there is $k \in \N^+$ such that $f^k(x) \in H$. Given $r \in \N^+$, the set $H$ is {\it $r$-forward-independent for $f$} if $f^k(x) \in H$ implies that $k > r$ for any $x \in H$ and $k \in \N^+$.

    \begin{thm}\label{thm:EquivalentCharFinAsDim}
        Let $X$ be a standard Borel space, and let $f : X \rightarrow X$ be an acyclic Borel function. Then the following are equivalent:
        \begin{enumerate}
            \item [(i)] $\asdim_B(\mcal{G}_f)$ is finite.
            \item [(ii)] $\asdim_B(\mcal{G}_f)= 1$.
            \item [(iii)] For each $r \in \N^+$, there is a Borel $r$-forward-independent hitting set for $f$.
        \end{enumerate}
    \end{thm}

\begin{remark}\label{rem:Bounded}
Condition (iii) in Theorem~\ref{thm:EquivalentCharFinAsDim} is satisfied if, for instance, $f$ is bounded-to-one.
This follows from \cite[Corollary~5.3]{kst1999}.
Furthermore, (iii) is also satisfied when the connectedness relation of $\mcal{G}_f$ admits a Borel transversal.
\end{remark}

Borel graphs with finite Borel asymptotic dimension enjoy certain definable combinatorial properties that are not shared by Borel hyperfinite graphs in general.
For instance, Brooks's theorem, a fundamental result in the theory of graph colorings, fails for Borel colorings \cite{marks2016}, even when the graph under consideration is Borel hyperfinite \cite{cjmst2020}.
However, a version of Brooks's theorem for Borel colorings, and in fact Borel versions of many local coloring problems that in general only admit solutions up to null sets, can be recovered for Borel graphs with finite Borel asymptotic dimension \cite{bw2023}.

As an application of the techniques used to prove Theorems~\ref{thm:main} and \ref{thm:EquivalentCharFinAsDim}, we derive a new complexity result for a homomorphism problem for the class of Borel \emph{directed graphs} (\emph{digraphs} for short) generated by a single Borel function.
Given a finite digraph $H$, we follow Thornton \cite{thornton2022} and write $\operatorname{CSP}_B(H)$ for the set (of codes) of Borel digraphs that admit a Borel homomorphism to $H$, and we also define $\operatorname{CSP}^{\operatorname{function}}_B(H)$ to be the set (of codes) of Borel digraphs \emph{generated by a single Borel function} that admit a Borel homomorphism to $H$.
Note that when investigating $\operatorname{CSP}^{\operatorname{function}}_B(H)$, we can without loss of generality assume that $H$ is \emph{sinkless}, i.e., that every vertex of $H$ has at least one outgoing edge.

Recall that a finite digraph $H$ is \emph{ergodic} if there are $v\in V(H)$ and $k_0\in \mathbb{N}$ such that, for every $k\ge k_0$, there is a directed path of length $k$ that starts and ends at $v$.\footnote{Here, the motivation for the term ``ergodic'' originates in the theory of Markov chains, where a transition digraph is said to be ergodic if it is irreducible and aperiodic, see, for example, \cite{Norris1997}.
In our context, the term is used for digraphs that contain a strongly connected component that is ergodic in the stochastic sense.}
A \emph{loop} is a directed edge of the form $(v, v)$ for some $v \in V(H)$. Note that any digraph containing a loop is ergodic.

    \begin{thm}\label{thm:Complexity}
        Let $H$ be a finite sinkless digraph.
        \begin{enumerate}
            \item [(i)] The digraph $H$ contains a loop if and only if $\operatorname{CSP}^{\operatorname{function}}_B(H)$ contains all digraphs generated by a single Borel function.
            In particular, the set $\operatorname{CSP}^{\operatorname{function}}_B(H)$ is $\mbf{\Pi}^1_1$.
            \item [(ii)] If $H$ is ergodic and has no loops, then $\operatorname{CSP}^{\operatorname{function}}_B(H)$ is $\mbf{\Sigma}^1_2$-complete.
            \item [(iii)] If $H$ is not ergodic, then $\operatorname{CSP}^{\operatorname{function}}_B(H)$ is $\mbf{\Pi}^1_1$.
        \end{enumerate}
    \end{thm}

    \begin{figure}
        \centering
        \includegraphics[width=0.6\textwidth]{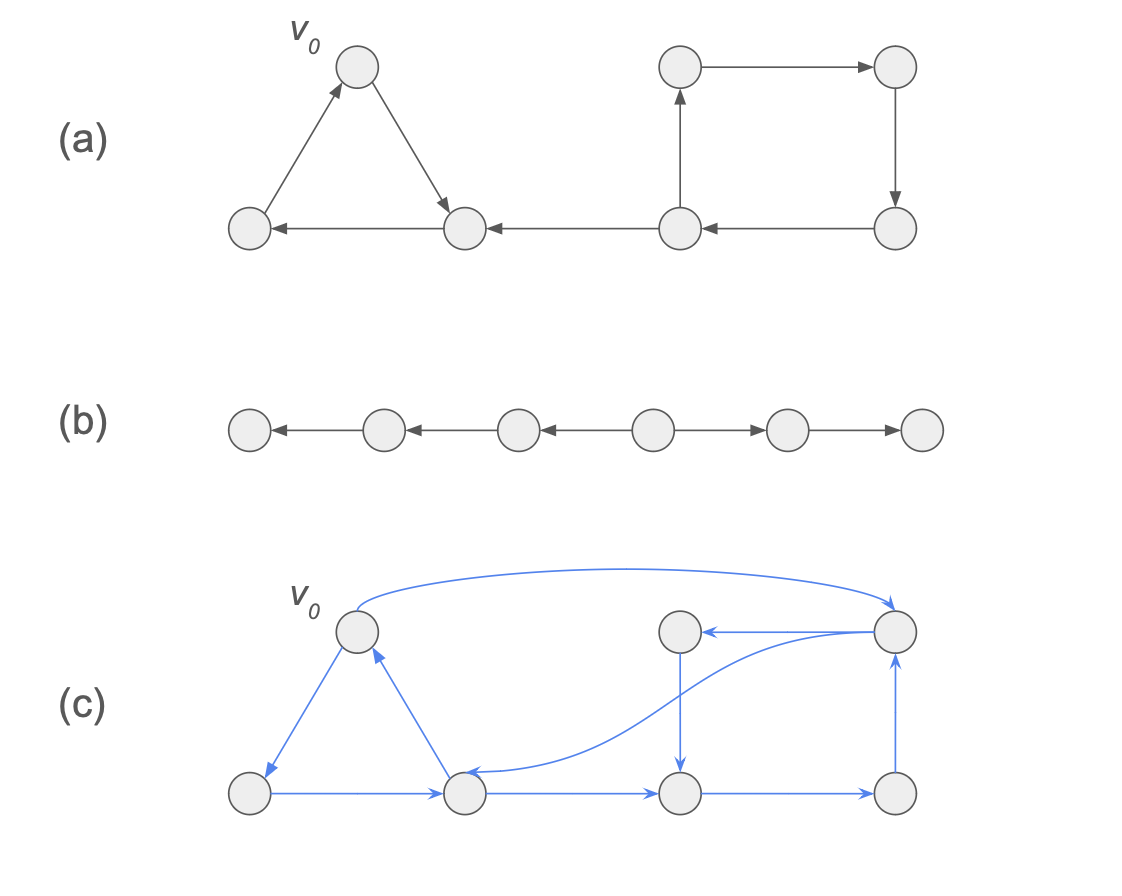}
        \caption{(a) A directed graph $H$ such that $\operatorname{CSP}^{\operatorname{function}}_B(H)$ is $\mbf{\Pi}^1_1$. (b) An abstract walk $p$. (c) The directed graph $H^p$, which has the property that $\operatorname{CSP}^{\operatorname{function}}_B(H^p)$ is $\mbf{\Sigma}^1_2$-complete, showing that $\operatorname{CSP}_B(H)$ is $\mbf{\Sigma}^1_2$-complete as well}
        \label{fig:h_p_graph}
    \end{figure}

\begin{remark}[Connections to CSPs]
\label{rem:FunCSPvsCSP}
    Theorem~\ref{thm:Complexity} is motivated by the classical theory of constraint satisfaction problems (CSPs); see, for example, \cite{brady2025} for a comprehensive overview of the subject.
    A (fixed-template) constraint satisfaction problem takes the following form: Given a finite structure $H$, how complicated is it to decide whether a finite structure in the same signature admits a homomorphism to $H$? The target structure $H$ is called the \emph{template}, and the set of finite structures admitting a homomorphism to $H$ is denoted $\operatorname{CSP}(H)$. 
    The well-known \emph{CSP dichotomy conjecture}, now a theorem, states that every CSP is either solvable in polynomial time or \textsf{NP}-complete.
    
    In their foundational paper, Feder and Vardi \cite{fv1999} showed that every CSP is polynomial-time equivalent to a digraph CSP, i.e., a CSP whose template is a digraph; therefore, resolving the full conjecture amounts to resolving it just for digraph CSPs.
    A special case of the CSP dichotomy in the case when the template is a \emph{sinkless and sourceless} digraph was established by Barto, Kozik, and Niven \cite{bkn2009}.
    Later, the full CSP dichotomy conjecture was confirmed independently by Bulatov \cite{bulatov2017} and by Zhuk \cite{zhuk2017,zhuk2020}.

    Recently, Thornton~\cite{thornton2022} initiated the study of the complexity of CSPs in measurable combinatorics.
    More precisely, given a finite digraph~$H$, we may consider, as above, the set $\operatorname{CSP}_B(H)$ (of codes) of Borel digraphs that admit a Borel homomorphism to~$H$.
    Among other results, Thornton~\cite{thornton2022} proved a Borel analogue of the CSP dichotomy for sinkless and sourceless digraphs; in particular, when $H$ is sinkless and sourceless, $\operatorname{CSP}_B(H)$ is either $\mathbf{\Pi}^1_1$ or $\mathbf{\Sigma}^1_2$-complete.
    Moreover, these two cases correspond exactly to the classical ones: $\operatorname{CSP}_B(H)$ is $\mathbf{\Pi}^1_1$ when $\operatorname{CSP}(H)$ is solvable in polynomial time, and $\operatorname{CSP}_B(H)$ is $\mathbf{\Sigma}^1_2$-complete when $\operatorname{CSP}(H)$ is \textsf{NP}-complete.
\end{remark}

\begin{remark}[$\operatorname{CSP}_B(H)$ vs $\operatorname{CSP}^{\operatorname{function}}_B(H)$]
Note that, when working with $\operatorname{CSP}^{\operatorname{function}}_B(H)$, we do not impose restrictions on the templates themselves; instead, we consider a subclass of digraphs on which the existence of Borel homomorphisms into $H$ is queried. For any digraph $H$, since $\mathbf{\Sigma}^1_2$ is an upper bound for the complexity of $\operatorname{CSP}_B(H)$, if $\operatorname{CSP}^{\operatorname{function}}_B(H)$ is $\mathbf{\Sigma}^1_2$-complete, then $\operatorname{CSP}_B(H)$ is $\mathbf{\Sigma}^1_2$-complete as well. The converse does not hold in general (see Figure~\ref{fig:h_p_graph} and its description for a counterexample).
However, it is possible that a weaker form of the converse holds; see the next paragraph and Section~\ref{sec:OpenProbs}.

Suppose now that $H$ is both sinkless and sourceless.
Let $p$ be an \emph{abstract walk}, i.e., a finite sequence of arrows as in Figure~\ref{fig:h_p_graph}(b).
Define $H^p$ to be the \emph{$p$-power of $H$}, i.e., the directed graph on $V(H)$ where there is a directed edge from $x$ to $y$ if there is a realization in $H$ of the walk $p$ from $x$ to $y$ (see Figure~\ref{fig:h_p_graph}(c) for an example).
Since $H$ is sinkless and sourceless, it can be easily seen that $H^p$ is also sinkless and sourceless.
If $\operatorname{CSP}^{\operatorname{function}}_B(H^p)$ is $\mbf{\Sigma}^1_2$-complete for some walk $p$, then so is $\operatorname{CSP}_B(H)$; this follows from the fact that, for any function $f$, the directed graph $\overrightarrow{\mathcal{G}}_f$ generated by $f$ admits a (Borel) homomorphism to $H^p$ if and only if the directed graph formed from $\overrightarrow{\mathcal{G}}_f$ by replacing each directed edge with the path $p$ admits a (Borel) homomorphism to $H$ (we note that this construction is definable in the codes).
We leave as an open problem whether the converse holds; see Section~\ref{sec:OpenProbs}.
\end{remark}

    \begin{remark}[Connections to the \LOCAL model]
    Theorem~\ref{thm:Complexity} provides a connection with the \LOCAL model of distributed computing and also serves as a specific instance of an interesting general phenomenon:
    For ``flexible'' problems such as $3$-coloring, there is a non-trivial class of graphs for which it is easy to produce solutions in a distributed way, but it is difficult to determine, for general graphs, whether a solution exists; while for ``rigid'' problems such as $2$-coloring, it is impossible to construct solutions on any non-trivial class of graphs, but it is easy to check, for general graphs, whether a solution exists.

    To be precise, recall first that any \emph{locally checkable labeling (LCL)} problem on oriented paths can be thought of as a digraph homomorphism problem $\Pi_H$ to a finite sinkless digraph $H$ \cite{bhk+2017}.
    The classification of LCL problems on oriented paths in the \LOCAL model \cite{bhk+2017} states that:
        \begin{enumerate}
            \item [(i)] If $H$ contains a loop, then $\Pi_H$ can be solved in $0$ rounds of communication in the \LOCAL model.
            \item [(ii)] If $H$ is ergodic and has no loops, then $\Pi_H$ requires $\Theta(\log^*(n))$ rounds of communication in the \LOCAL model.
            \item [(iii)] If $H$ is not ergodic, then $\Pi_H$ requires $\Theta(n)$ rounds of communication in the \LOCAL model.
        \end{enumerate}
    Comparing with Theorem~\ref{thm:Complexity}, we see that the existence of an efficient non-trivial \LOCAL algorithm for solving $\Pi_H$ on oriented paths is equivalent to $\operatorname{CSP}^{\operatorname{function}}_B(H)$ being a $\mbf{\Sigma}^1_2$-complete set.
    If no efficient \LOCAL algorithm exists, then $\operatorname{CSP}^{\operatorname{function}}_B(H)$ is $\mbf{\Pi}^1_1$.
    \end{remark}

The paper is organized as follows. In Section~\ref{sec:Prelim}, we provide some preliminaries. In Section~\ref{sec:Complexity}, we prove the main complexity result that we use throughout the rest of the paper. The proof of Theorem~\ref{thm:Complexity} is given in Section~\ref{sec:DigraphHom}. In Section~\ref{sec:FinAsDim}, we prove Theorems~\ref{thm:main} and \ref{thm:EquivalentCharFinAsDim}. We conclude with some open questions in Section~\ref{sec:OpenProbs}.

\section{Preliminaries}
\label{sec:Prelim}

\newcommand{\dig}{\overrightarrow{\mathcal{G}}_f}
\newcommand{\digS}{\overrightarrow{\mathcal{G}}_S}

Let $X$ be a standard Borel space.
Given a Borel function $f:X\to X$, we write $\mcal{G}_f$ for the Borel graph generated by $f$ and $\dig$ for the Borel \emph{directed} graph (digraph) generated by $f$.

Recall that $[\N]^{\N}$ denotes the space of infinite subsets of $\N$. Each point $x \in [\N]^{\N}$ may be identified with its increasing enumeration $x = (x(n))_{n \in \N}$.
The \emph{shift function} $S:[\N]^\N\to [\N]^\N$ is defined by
$$S(x)(n)=x(n+1) \text{ for every } n\in \N,$$
and the \emph{shift graph} $\mcal{G}_S$ is the (undirected) graph on $[\N]^{\N}$ generated by $S$.

For every $x,y\in [\N]^\N$, we write $y\le^\infty x$ if the set $\{n\in \N:y(n)\le x(n)\}$ is infinite. We define
$$\Dom=\{(x,y)\in [\N]^\N\times [\N]^\N:y\le^\infty x\}.$$

\medskip

Given sets $X, Y$ and a set $B\subseteq X\times Y$, we write $B_x=\{y\in Y:(x,y)\in B\}$ for each $x \in X$.

Let $R$ be a Borel relation on $X$, let $S$ be a Borel relation on a standard Borel space $Y$, and assume that $R$ and $S$ have the same arity $d \in \N$.
A Borel map $\varphi:X\to Y$ is a \emph{homomorphism from $(X,R)$ to $(Y,S)$}, or from $R$ to $S$ for short, if 
$$(x_0,\dots,x_{d-1})\in R \ \Rightarrow \ (f(x_0),\dots,f(x_{d-1}))\in S$$
for every $(x_i)_{i=0}^{d-1}\in X^d$.

\medskip

A finite digraph $H=(V(H),E(H))$ is \emph{sinkless} if every vertex has an outgoing edge and \emph{sourceless} if every vertex has an ingoing edge.
A \emph{directed path} in $H$ is a sequence of vertices $(v_0,\dots v_k)\subseteq V(H)$ such that $(v_i,v_{i+1})\in E(H)$ for every $0\le i\le k-1$.
A \emph{strong connectivity component of $H$} is an inclusion-maximal set $A\subseteq V(H)$ such that, for any vertices $v, w \in A$, there is a directed path from $v$ to $w$ in $H \rest A$.
We say that $H$ is \emph{strongly connected} if $V(H)$ is a strong connectivity component. 
Note that every finite sinkless digraph contains at least one strong connectivity component since it contains some vertex $v$ for which there is a directed path from $v$ to itself.

\medskip

\noindent
{\bf ${\bf \Sigma}^1_2$-completeness.}
The projective hierarchy provides a general framework for the study of complexity problems in measurable combinatorics.
Recall that, given a Polish space $X$, a set $A \subseteq X$ is $\mbf{\Sigma}^1_2$ if it is the projection of a ${\bf \Pi}^1_1$ (co-analytic) set (see \cite[Chapter~V]{kechris1995}).
For example, it is easy to see that the set of (codes of) locally finite Borel graphs having finite Borel chromatic number is $\mbf{\Sigma}^1_2$.
More generally, given any local graph coloring problem, the set of (codes of) locally finite Borel graphs admitting a solution to the problem is $\mbf{\Sigma}^1_2$.
We refer the reader to \cite{moschovakis2009} and \cite[Section~1]{fsv2024} for an overview of the coding method.

\begin{defn}[${\bf \Sigma}^1_2$-completeness]
    Let $A$ be a subset of a Polish space $X$.
    We say that $A$ is \emph{${\bf \Sigma}^1_2$-hard} if for every Polish space $Y$ and every ${\bf \Sigma}^1_2$ set $B\subseteq Y$, there is a Borel map $f:Y\to X$ such that $f^{-1}(A) = B$, i.e., $f(y)\in A$ if and only if $y\in B$.
    A set $A\subseteq X$ is called \emph{${\bf \Sigma}^1_2$-complete} if it is simultaneously ${\bf \Sigma}^1_2$-hard and ${\bf \Sigma}^1_2$.
\end{defn}

As mentioned in the introduction, Todor\v{c}evi\'{c} and Vidny\'{a}nszky \cite{tv2021} showed that the set of codes of locally finite Borel graphs that have finite Borel chromatic number is a ${\bf \Sigma}^1_2$-complete set.
In our arguments, we use a general result for proving ${\bf \Sigma}^1_2$-completeness that builds on \cite{tv2021} and was derived recently by Frisch, Shinko and Vidny\'{a}nszky \cite[Theorem~3.1]{fsv2024}.
We refer the reader to \cite[Section~1]{fsv2024} for the definition of Borel $L$-structure.

\begin{thm}[Theorem~3.1, \cite{fsv2024}]\label{thm:FSV}
    Let $\mathcal{H}$ be a Borel $L$-structure on some Polish space $Z$, and assume that there exist a Borel $L$-structure $\mathcal{G}$ on $[\mathbb{N}]^\mathbb{N}$ that does not admit a Borel homomorphism to $\mathcal{H}$ and a Borel map $\Phi: \Dom\to Z$ so that, for each $x$, we have that $\Phi_x$  is a homomorphism from $\mathcal{G}\upharpoonright \Dom_x$ to $\mathcal{H}$. 
    Then the Borel $L$-structures that admit a Borel homomorphism to $\mathcal{H}$ form a ${\bf \Sigma}^1_2$-complete set.
\end{thm}

In fact, we use a slight strengthening of Theorem~\ref{thm:FSV} that follows from the proof of \cite[Theorem~3.1]{fsv2024}.
Given an $L$-structure $\mathcal{G}$ on $[\mathbb{N}]^\mathbb{N}$, write $\mathcal{G}'$ for the $L$-structure on $\N^{\mathbb{N}}\times [\mathbb{N}]^\mathbb{N}$ where each vertical section is a copy of $\mathcal{G}$ and no additional elements are related.

\begin{thm}\label{thm:FSVversionTwo}
    Let $\mathcal{H}$ be as in Theorem~\ref{thm:FSV}.
    Then there is a Borel set $B\subseteq \mathbb{N}^\mathbb{N}\times \mathbb{N}^\mathbb{N}\times [\mathbb{N}]^\mathbb{N}$ such that the set
   \begin{equation*}
        \{\sigma\in \N^\N: \text{ there is a Borel homomorphism from } (\mcal{G}' \rest B_{\sigma}) \text{ to } \mcal{H} \}
    \end{equation*}
    is ${\bf \Sigma}^1_2$-complete.
\end{thm}

\section{Complexity of forward-independent hitting sets}
\label{sec:Complexity}

In this section, we use Theorem~\ref{thm:FSVversionTwo} to deduce the following theorem, which immediately implies that the set (of codes) of Borel functions admitting Borel forward-independent hitting sets is $\mbf{\Sigma}^1_2$-complete.
Our main result, Theorem~\ref{thm:main}, is a corollary of Theorem~\ref{thm:ComplexityFIH} together with Theorem~\ref{thm:EquivalentCharFinAsDim}, which is proved in Section~\ref{sec:FinAsDim}.

\begin{thm}\label{thm:ComplexityFIH}
    There are a standard Borel space $X$ and a finite-to-one acyclic Borel function
    $$f:\mathbb{N}^\mathbb{N}\times X\to \mathbb{N}^\mathbb{N}\times X$$
    with the property that, for every $\sigma\in \mathbb{N}^\mathbb{N}$, there exists $f_\sigma: X\to X$ such that $$f(\sigma,x)=(\sigma,f_\sigma(x)) \text{ for every } (\sigma,x)\in \mathbb{N}^\mathbb{N}\times X,$$
    and
\begin{equation}\label{eq:AllComplete}
    \{\sigma\in \mathbb{N}^\mathbb{N}: \forall r\in \mathbb{N}^+ \ f_\sigma \ \text{\normalfont admits a Borel } r \text{\normalfont-forward-independent hitting set} \}
\end{equation}
is $\mbf{\Sigma}^1_2$-complete.

Additionally, for every $r\in \mathbb{N}^+$, there is a Borel function $f$ as above such that
\begin{equation}\label{eq:OneComplete}
    \{\sigma\in \mathbb{N}^\mathbb{N}:  f_\sigma \ \text{\normalfont admits a Borel } r \text{\normalfont-forward-independent hitting set} \}
\end{equation}
is $\mbf{\Sigma}^1_2$-complete.
\end{thm}

For each $r \in \N^+$, let $\mcal{D}_r$ be the digraph on $\N$ such that there is a directed edge from $k$ to $\ell$ if and only if $k > 0$ and $\ell = k - 1$ or $k = 0$ and $\ell \geq r$ (see Figure~\ref{fig:hitting_graph}).

    \begin{figure}
        \centering
        \includegraphics[width = 0.7\textwidth]{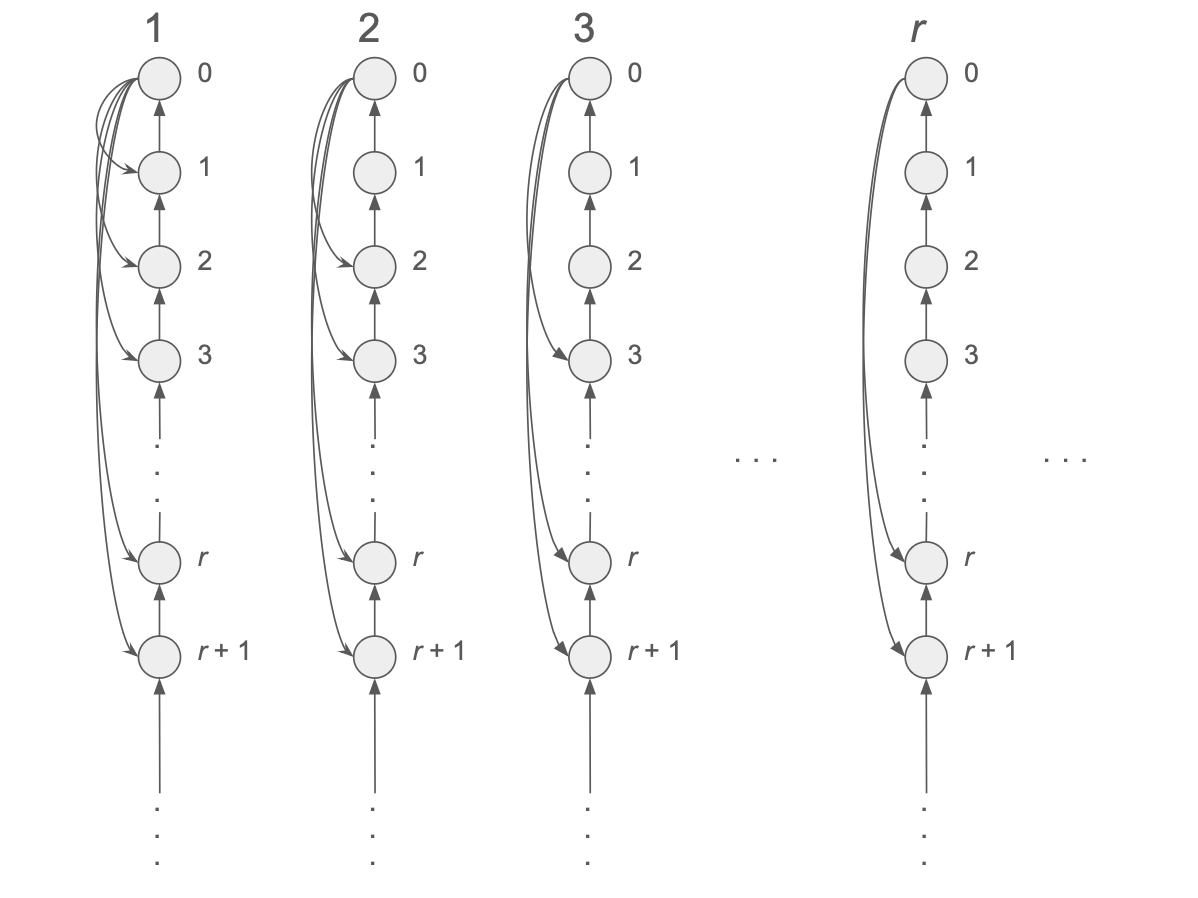}
        \caption{The structure $\mcal{D}_r$ on $\N$ for various values of $r \in \N^+$}
        \label{fig:hitting_graph}
    \end{figure}

    \begin{lem}
    \label{lem:hitting_to_L_0}
        Let $f:X\to X$ be an acyclic Borel function, and let $r\in \mathbb{N}^+$.
        Then the following are equivalent:
        \begin{enumerate}
            \item [(i)] There is a Borel $r$-forward-independent hitting set for $f$.
            \item [(ii)] There is a Borel map $\varphi:X\to \mathbb{N}$ that is a homomorphism from $(X,\dig)$ to $(\mathbb{N},\mcal{D}_r)$.
        \end{enumerate}
    \end{lem}

    \begin{proof}  
        (i) $\implies$ (ii). Let $H \subseteq X$ be a Borel $r$-forward-independent hitting set for $f$.
        Define a Borel map $\varphi : X \rightarrow \N$ by letting $\varphi(x)$ be the least $k \in \N$ such that $f^k(x) \in H$.
        Note that $\varphi(x) = 0$ if and only if $x \in H$.
        To see that $\varphi$ is a homomorphism, let $x\in X$ and observe that there is a directed edge from $\varphi(x)$ to $\varphi(f(x))$;
        indeed, if $\varphi(x) > 0$, then $\varphi(f(x)) = \varphi(x) - 1$, and if $\varphi(x) = 0$, then $x \in H$, which implies that $\varphi(f(x))\ge r$ since $H$ is $r$-forward-independent. 

        (ii) $\implies$ (i). Let $\varphi : \overrightarrow{\mcal{G}}_f \rightarrow \mcal{D}_r$ be a Borel homomorphism. Set $H = \varphi^{-1}(\{0 \})$.
        Then clearly $H$ is Borel.
        To see that $H$ is $r$-forward-independent, suppose $x, y \in H$ with $f^k(x) = y$ for some $k > 0$.
        Then $\varphi(x) = 0$, so $\varphi(f^i(x)) > 0$ whenever $1 \leq i \leq r$. Since $\varphi(f^k(x)) = 0$, we have $k > r$.
        To see that $H$ is hitting for $f$, let $x \in X$ and observe that $f^{\varphi(x)}(x)\in H$ as $\varphi$ is a homomorphism.
    \end{proof}

Now we are ready to prove Theorem~\ref{thm:ComplexityFIH}.
The proof strategy is as follows: We show first that, while there are no Borel forward-independent hitting sets for $S$ on the entirety of $[\N]^{\N}$, such sets can be constructed for $S$ on the \emph{non-dominating} subsets of $[\N]^{\N}$, i.e., the sets of the form $\Dom_x$ where $x \in [\N]^{\N}$; furthermore, this construction can be done in a uniform Borel way.
This then allows us to apply \cite[Theorem~3.1]{fsv2024}, which gives the $\mbf{\Sigma}^1_2$-completeness of the sets (3.2) and (3.3).

\begin{proof}[Proof of Theorem~\ref{thm:ComplexityFIH}]

We first fix $r \in \N^+$ and prove that there are a standard Borel space $X$ and a Borel function $f:\mathbb{N}^\mathbb{N}\times X\to \mathbb{N}^\mathbb{N}\times X$ as in the statement of Theorem~\ref{thm:ComplexityFIH} such that the set~\eqref{eq:OneComplete} is $\mbf{\Sigma}^1_2$-complete. At the end of the proof, we explain how to re-select $X$ and $f$ so that the set~\eqref{eq:AllComplete} is $\mbf{\Sigma}^1_2$-complete.

Recall that $\digS$ is the Borel digraph on $[\N]^{\N}$ induced by the shift function $S$ on $[\N]^{\N}$.
Note that there is \emph{no} Borel map $\varphi:[\N]^\N\to \mathbb{N}$ that is a homomorphism from $([\N]^\N,\digS)$ to $(\N,\mcal{D}_r)$.
This follows, for example, from \cite[Example~3.2]{kst1999}, since the existence of a Borel $r$-forward-independent hitting set for $S$ would imply that the Borel chromatic number of $\mcal{G}_S$ is finite.

In order to apply Theorem~\ref{thm:FSVversionTwo}, we must show that there is a Borel map $\varphi_r : \Dom \to \N$ such that, for each $x \in [\N]^{\N}$, the restriction $\varphi_r \rest \Dom_x$ is a homomorphism from $\overrightarrow{\mcal{G}}_S \rest \Dom_x$ to $\mcal{D}_r$, where $\digS \rest \Dom_x$ denotes the sub-digraph of $\digS$ induced by $\Dom_x$.

Set $x(-r) = 0$ for all $x \in [\N]^{\N}$. Define $\mcal{H}_r \subseteq ([\N]^{\N})^2$ by $(x, y) \in \mcal{H}_r$ if and only if $\vert [x(2rn - r), x(2rn + r)) \cap y \vert \equiv_{2r} 0$, where $n \in \N$ is minimal such that $[x(2rn - r), x(2rn + r)) \cap y \neq \emptyset$.
Observe that $\mcal{H}_r$ is Borel.

    \begin{claim}\label{cl:WellDefined}
        If $(x, y) \in \Dom$, then there is $n \in \N$ such that $\vert [x(2rn - r), x(2rn + r)) \cap y \vert \geq 2r$. 
    \end{claim}
    \begin{proof}
        Let $(x,y)\in \Dom$ and assume for contradiction that there is no $n\in \N$ such that $\vert[x(2rn - r), x(2rn + r)) \cap y \vert\ge 2r$.
        For each $n \in \N$, let $\ell_n\in \N$ be minimal such that $y(\ell_n)\ge x(2rn-r)$; it follows from the assumption that $\ell_{n+1}<\ell_n+2r$.
        Hence, $\ell_n \leq 2rn - n$ for each $n \in \N$. So, there is $n_0\in \N$ such that $\ell_{n}< 2rn-3r$ for every $n\ge n_0$.
        Let $k\ge \ell_{n_0}$, and let $n\ge n_0$ be minimal such that $k\ge \ell_n$.
        Then we have
        $$y(k)\ge y(\ell_n)\ge x(2rn-r)> x(k)$$
        as $2rn-r> \ell_n+2r > \ell_{n+1} \ge k$, which is a contradiction as $y\le^\infty x$.
    \end{proof}

Now for each $(x, y) \in \Dom$, define $\varphi_r(x, y)$ to be the least $k \in \N$ such that $S^k(y) \in (\mcal{H}_r)_x$.
By Claim~\ref{cl:WellDefined}, we have that $\varphi_r$ is well-defined and it is easy to check that it is Borel as $\mathcal{H}_r$ is Borel.
To see that $\varphi_r \rest \Dom_x$ is a homomorphism from $\overrightarrow{\mcal{G}}_S \rest \Dom_x$ to $\mcal{D}_r$ for each $x \in [\N]^{\N}$, note that, if $y, y' \in \Dom_x$ and $S(y) = y'$, then $\varphi_r(x, y) = 0$ implies $\varphi_r(x, S(y)) \geq r$ since
$$\vert [x(2rn - r), x(2rn + r)) \cap S(y) \vert \equiv_{2r} 2r-1,$$
where $n \in \N$ is minimal such that $[x(2rn - r), x(2rn + r)) \cap y \neq \emptyset$.
If $\varphi_r(x, y) > 0$, then $\varphi_r(x, S(y)) = \varphi_r(x, y) - 1$ by the definition.

Now we are ready to finish the first part of the proof.
Note that the assumptions of Theorem~\ref{thm:FSVversionTwo} are satisfied by the construction of $\varphi_r$ above.
Hence, there is a Borel set $B\subseteq \N^\N\times \N^\N\times [\N]^\N$ such that 
    \begin{equation}\label{eq:TheirComplex}
        \{\sigma\in \N^\N: \text{ there is a Borel homomorphism from } (\overrightarrow{\mcal{G}}'_S \rest B_{\sigma}) \text{ to } \mcal{D}_r \}
    \end{equation}
is $\mbf{\Sigma}^1_2$-complete, where $\overrightarrow{\mcal{G}}'_S$ is the digraph on $\N^{\N} \times [\N]^{\N}$ such that there is a directed edge from $(\gamma, x)$ to $(\delta, y)$ if and only if $\gamma = \delta$ and $S(x) = y$.
Let
$$Y=\{(\sigma,\gamma,x)\in B:(\sigma,\gamma,S(x))\not\in B\}$$
and observe that $Y$ is a Borel set.
Set also $Z=(\N^\N\times \N^\N\times [\N]^\N)\setminus B$.
We define
$$X=\bigsqcup_{\ell\in \mathbb{N}} (\N^\N\times [\N]^\N\times \{\ell\})$$
and a function $f$ on $\mathbb{N}^\N\times X$ as follows.
If $(\sigma,(\gamma,x,\ell))\in \N^{\N} \times X$ has the property that $(\sigma,\gamma,x)\not\in Y\cup Z$, then we set $f(\sigma,(\gamma,x,\ell))=(\sigma,(\gamma,S(x),\ell))$.
Otherwise, we have that $(\sigma,\gamma,x)\in Y\cup Z$ and we let $f(\sigma,(\gamma,x,\ell))=(\sigma,(\gamma,x,\ell+1))$.
Then $f: \N^{\N} \times X\to \N^{\N} \times X$ is a finite-to-one acyclic Borel function. 
Fix $\sigma \in \N^{\N}$, and let $f_{\sigma}$ be the function on $X$ such that $f(\sigma, (\gamma, x, \ell)) = (\sigma, f_{\sigma}(\gamma, x, \ell))$ for all $(\gamma, x, \ell) \in X$.
Note that the collection $\{f_\sigma\}_{\sigma\in \mathbb{N}^\mathbb{N}}$ is well-defined.

To complete this part of the proof, we show that the sets \eqref{eq:TheirComplex} and \eqref{eq:OneComplete} coincide.
Define a set $T \subseteq X$ by $(\gamma, x, \ell) \in T$ if $\ell = 0$ and $(\sigma, \gamma, x) \in Y \cup Z$, and let $\mcal{C}$ denote the collection of connected components of the graph $\mcal{G}_{f_{\sigma}}$ on $X$ generated by $f_{\sigma}$ containing a point $(\gamma, x, \ell) \in X$ such that $(\sigma, \gamma, x) \in Y \cup Z$. Then $T$ is a Borel transversal for the connectedness relation of $\mcal{G}_{f_{\sigma}} \rest \bigcup \mcal{C}$. Therefore, since $f_{\sigma}$ is acyclic, we have that $f_{\sigma} \rest \bigcup \mcal{C}$ admits a Borel $r$-forward-independent hitting set by Remark~\ref{rem:Bounded}. It follows that the sets \eqref{eq:TheirComplex} and \eqref{eq:OneComplete} coincide, as required.

We now explain how to modify the proof to obtain $X$ and $f$ so that \eqref{eq:AllComplete} is $\mbf{\Sigma}^1_2$-complete: Consider a language consisting of countably many binary relation symbols $(\mcal{R}_r)_{r \in \N^+}$. For each $r \in \N^+$, interpret $\mcal{R}_r$ in $\N$ as the edge relation of $\mcal{D}_r$, and apply Theorem~\ref{thm:FSVversionTwo} to obtain a Borel set $B \subseteq \N^{\N} \times \N^{\N} \times [\N]^{\N}$ such that
\begin{equation}
        \{\sigma\in \N^\N: \forall r\in \N^+ \text{ there is a Borel homomorphism from } (\overrightarrow{\mcal{G}}'_S \rest B_{\sigma}) \text{ to } \mcal{D}_r \}
    \end{equation}
is $\mbf{\Sigma}^1_2$-complete. The rest of the proof proceeds in the same way.
\end{proof}

\begin{remark}
\label{rem:Alex}
    In personal communication, Alex Kastner pointed out that the Borel set $B \subseteq \N^{\N} \times \N^{\N} \times [\N]^{\N}$ obtained from the application of Theorem~\ref{thm:FSVversionTwo} can be chosen so that, for each $\sigma, \gamma \in \N^{\N}$, the section $B_{\sigma, \gamma}$ is $\mcal{G}_S$-invariant. This allows one to define $X = \N^{\N} \times [\N]^{\N}$ and $f(\sigma, \gamma, x) = (\sigma, \gamma, S(x))$ directly in the proof of Theorem~\ref{thm:ComplexityFIH}.
\end{remark}

\section{The classification of $\operatorname{CSP}^{\operatorname{function}}_B(H)$}
\label{sec:DigraphHom}

This section is devoted to a proof of Theorem~\ref{thm:Complexity}.
As in the proof of of Theorem~\ref{thm:ComplexityFIH}, we utilize \cite[Theorem~3.1]{fsv2024}.

Recall that the set of codes for Borel functions on a fixed standard Borel space $X$ is $\mbf{\Pi}^1_1$ \cite[Lemma~A.3]{hkm2024}.
For the remainder of the section, we fix a finite sinkless digraph $H$.

\begin{proof}[Proof of Theorem~\ref{thm:Complexity}]
(i). Let $X$ be standard Borel, and let $f$ be a Borel function on $X$. If $(v,v)$ is a loop in $H$, then the function $\psi : X \to V(H)$ defined by $\psi(x) = v$ for all $x \in X$ is clearly a Borel digraph homomorphism.

Conversely, if $\operatorname{CSP}^{\operatorname{function}}_B(H)$ contains all Borel digraphs, then in particular it contains the shift digraph $\overrightarrow{\mathcal{G}}_S$.
Let $\psi:[\N]^\N\to V(H)$ be a Borel homomorphism from $\overrightarrow{\mathcal{G}}_S$ to $H$.
By the Galvin--Prikry theorem, since $V(H)$ is finite, there are $v\in V(H)$ and $y\in [\N]^\N$ such that $\psi(S^k(y))=v$ for every $k\in \N$.
Consequently, $(v,v)$ is a loop in $H$, and this completes the proof.

\medskip

(ii). Assume that $H$ is an ergodic digraph with no loops.
Then by the previous paragraph, we have $\overrightarrow{\mathcal{G}}_S\not\in \operatorname{CSP}^{\operatorname{function}}_B(H)$.
The next result, Proposition~\ref{pr:FlexibleLCL}, guarantees that \cite[Theorem~3.1]{fsv2024} can be used as in the proof of Theorem~\ref{thm:ComplexityFIH} to obtain an acyclic Borel function $f=(f_\sigma)_{\sigma\in{\N^\N}}:\N^\N\times X\to \N^\N\times X$ such that 
\begin{equation}\label{eq:CompleteErgodic}
    \{\sigma\in \mathbb{N}^\mathbb{N}:  \overrightarrow{\mathcal{G}}_{f_\sigma} \ \text{\normalfont admits a Borel homomorphism to } H\}
\end{equation}
is $\mbf{\Sigma}^1_2$-complete, thereby giving (ii).

    \begin{prop}\label{pr:FlexibleLCL}
        Let $H$ be a finite sinkless digraph that is ergodic and has no loops.
        Then there is $\ell_0\in \N$ such that, for every acyclic Borel function $f$ on a standard Borel space $X$ that admits a Borel $\ell_0$-forward-independent hitting set $A$, there is a Borel digraph homomorphism from $\dig$ to $H$.    
    \end{prop}
    \begin{proof}
        Let $v_0\in V(H)$ and $k_0\in \N$ be such that for every $k\ge k_0$, there is a directed path from $v_0$ to $v_0$ of length exactly $k$.
        We may, without loss of generality, and after possibly passing to the strong connectivity component of $v_0$ in $H$, assume that $H$ is strongly connected.
        We claim that there is $\ell_0\in \N$ such that, for every $\ell\ge \ell_0$ and any $w \in V(H)$, there is a directed path of length $\ell$ from $v_0$ to $w$.
        Indeed, by strong connectivity of $H$, for any $w\in V(H)$ there is some path from $v_0$ to $w$, and this path can be pre-composed with a path from $v_0$ to $v_0$ of any length $k\ge k_0$.

        Assign to each $z\in A$ a set $A_z=\bigcup_{j=1}^{\ell_0} f^{-j}(z)$.
        Note that, if $z, z' \in A$ and $z \neq z'$, then $A_z \cap A_{z'} = \emptyset$; indeed, if $y\in A_z\cap A_{z'}$, then there are $1\le j\le j'\le \ell_0$ such that $f^j(y),f^{j'}(y)\in A$.
        In particular, $f^{j'-j}(f^j(y)),f^j(y)\in A$ which can only happen if $j=j'$ as $A$ is $\ell_0$-forward-independent, implying that $z=z'$. Now define $Z = \bigsqcup_{z \in A} A_z$ and note that $Z$ is Borel and disjoint from $A$.
        
        For every $x\in X\setminus Z$, define $k(x)\in \N^+$ to be minimal such that $f^{k(x)}(x)\in Z$.
        Note that $k(x)$ is well-defined as $A$ is hitting for $f$.
        Fix a directed path $C$ from $v_0$ to $v_0$; for each $v \in C$, define $m(v)$ to be the number of steps in $C$ from $v$ to $v_0$.
        Then, for each $x \in X \setminus Z$, define $\psi(x)$ to be the least $v \in C$ such that $m(v)\equiv_{|C|} k(x)$.
        Now we define $\psi$ on $Z$; first, for each $z\in A$, fix a directed path $D_z$ from $v_0$ to $\psi(z)$ of length $\ell_0$.
        For each $y \in Z$, there is a unique $z \in A$ such that $y \in A_z$. Define $\psi(y)$ to be the unique $w \in D_z$ such that, if $f^j(y)=z$, then $j$ is the number of steps in $D_z$ from $w$ to $\psi(z)$.
        Then $\psi$ is as required.
    \end{proof}

    (iii).
    We start with the following observation.

    \begin{claim}\label{cl:Observe}
        Let $f:X\to X$ be a Borel function (not necessarily acyclic).
        Then there is a Borel homomorphism $\psi$ from $\dig$ to $H$ if and only if there is a Borel homomorphism $\psi'$ from $\dig$ to a union of strong connectivity components of $H$.
    \end{claim}
    \begin{proof}
    For every strong connectivity component $A\subseteq V(H)$ of $H$, define
    $$X_A=\bigcup_{k\in \N} \bigcap_{j=k}^\infty f^{-j}(\psi^{-1}(A)).$$
    Observe that $X=\bigsqcup_A X_A$, where $A$ ranges over all strong connectivity components of $H$, is a decomposition of $X$ into  Borel sets that are $f$-invariant.

    Fix a strong connectivity component $A$ and assign to each $v\in A$ a sequence $(w(k,v))_{k\in \N}\subseteq A$ such that $w(0,v)=v$ and $(w(k+1,v),w(k,v))\in E(H)$ for every $k\in \N$.    
    Define 
    $$\psi_A(x)=w(k,v)\in A$$
    for every $x\in X_A$, where $k\in \N$ is minimal such that $f^k(x)\in \psi^{-1}(A)$ and $\psi(f^k(x))=v$.
    Then $\psi'=\bigsqcup_A \psi_A$ is the desired homomorphism.    
    \end{proof}

    By Claim~\ref{cl:Observe}, we may, without loss of generality, assume that $H$ is a disjoint union of strongly connected digraphs.
    In particular, $H$ is both sinkless and sourceless.
    By the assumption in (iii), none of the strong connectivity components of $H$ is ergodic; then by \cite[Corollary~6.3]{thornton2022}, $\operatorname{CSP}_B(H)$ is $\mbf{\Pi}^1_1$.
    Consequently, $\operatorname{CSP}^{\operatorname{function}}_B(H)$ is $\mbf{\Pi}^1_1$ as well.
\end{proof}

\section{Finite asymptotic dimension for graphs generated by a Borel function}\label{sec:FinAsDim}

In this section, we prove Theorems~\ref{thm:EquivalentCharFinAsDim} and \ref{thm:main}.
We begin by recalling the definition of \emph{Borel asymptotic dimension} for a Borel graph \cite[Definition~3.2]{cjmst2023}. In \cite{cjmst2023}, two separate definitions -- one in terms of coverings and the other in terms of equivalence relations -- are provided; these are shown to be equivalent for locally countable Borel graphs \cite[Lemma~3.1]{cjmst2023}. We prove that, for graphs generated by a single Borel function $f$, the two definitions remain equivalent even when $f$ is \emph{uncountable}-to-one.

\medskip

Fix a standard Borel space $X$. Since we work exclusively with graphs of the form $\mcal{G}_f$, where $f$ is a Borel function on $X$, we denote by $\rho_f$ the graph distance metric of $\mcal{G}_f$.
Given $t \in \N^+$ and $x \in X$, we write $B_t(x)$ for the $\rho_f$-ball of radius $t$ around $x\in X$.
Given $U\subseteq X$ and $r\in \N^+$, we define $\mathcal{F}_r(U)$ to be the equivalence relation on $U$ generated by the set of pairs $(x,y)\in U^2$ such that $\rho_f(x,y)\le r$.
Given a Borel equivalence relation $E$ on $X$, we say $E$ is \emph{uniformly $\rho_f$-bounded} (or merely \emph{uniformly bounded} if $\rho_f$ is understood) if there is $M > 0$ such that each $E$-class has $\rho_f$-diameter at most $M$.

\begin{defn}
\label{def:Asdim}
    Let $f:X\to X$ be a Borel function.
\begin{itemize}
    \item [(i)] The \emph{Borel asymptotic dimension of $\mcal{G}_f$ from coverings}, denoted $\asdim_B^{\text{cov}}(\mcal{G}_f)$, is equal to $d \in \N$ if $d$ is minimal such that, for every $r\in \N^+$, there are Borel sets $U_0, \dots, U_d$ covering $X$ such that $\mcal{F}_r(U_i)$ is uniformly $\rho_f$-bounded for every $i \leq d$.
    \item [(ii)] The \emph{Borel asymptotic dimension of $\mcal{G}_f$ from equivalence relations}, denoted $\asdim_B^{\text{eq}}(\mcal{G}_f)$, is equal to $d \in \N$ if $d$ is minimal such that, for every $r \in \N^+$, there is a uniformly $\rho_f$-bounded Borel equivalence relation $E$ on $X$ such that, for each $x\in X$, $B_r(x)$ meets at most $(d + 1)$ $E$-classes. 
\end{itemize}
\end{defn}

In Subsection~\ref{subsec:AsdimEquiv}, we show $\asdim^{\text{cov}}_B(\mcal{G}_f) = \asdim^{\text{eq}}_B(\mcal{G}_f)$ for any Borel function $f$, so that we may define $\asdim_B(\mcal{G}_f)$ to be $\asdim^{\text{cov}}_B(\mcal{G}_f)$ (equivalently, $\asdim^{\text{eq}}_B(\mcal{G}_f)$).

In Subsections~\ref{subsec:AsdimFIH} and \ref{subsec:FIHAsdim}, we demonstrate, for graphs of the form $\mcal{G}_f$, a connection between both the equivalence relation and covering definitions of finite Borel asymptotic dimension and the presence of Borel forward-independent hitting sets. In Subsection~\ref{subsec:AsdimEquiv}, we use this hitting-set characterization of finite Borel asymptotic dimension to prove Theorem~\ref{thm:EquivalentCharFinAsDim}, and in Subsection~\ref{subsec:MainThm}, we prove Theorem~\ref{thm:main}.

\subsection{From finite asymptotic dimension to hitting sets}
\label{subsec:AsdimFIH}

\begin{prop}\label{pr:CovToFIH}
    Let $f:X\to X$ be an acyclic Borel function, $A\subseteq X$ be a Borel hitting set for $f$ and $r\in\N^+$ be such that every equivalence class of $\mcal{F}_r(A)$ has finite diameter.
    Then there is a Borel $r$-forward-independent hitting set $H$ for $f$.
\end{prop}
\begin{proof}
    Let 
        \begin{align*}
            H = \{x \in A : f^j(x) \notin A \text{ for any $1 \leq j \leq r$} \}=A\setminus \bigcup_{j=1}^r f^{-j}(A).
        \end{align*}
    Observe that $H$ is a Borel $r$-forward-independent set.
    It remains to show that $H$ is hitting for $f$.
    As $A$ is hitting for $f$, we have that for every $x\in X$ there is $k\in \N$ such that $f^k(x)\in A$.
    Let $\ell\in \N$ be maximal so that $f^{k+\ell}(x)$ and $f^k(x)$ are $\mcal{F}_r(A)$-equivalent.
    Note that such $\ell\in\N$ exists as the $\mcal{F}_r(A)$-equivalence class of $f^k(x)$ has finite diameter.
    It follows from the definition of $\mcal{F}_r(A)$ that $f^{k+\ell}(x)\in H$, which finishes the proof.
\end{proof}

\begin{prop}\label{pr:EqRelToHit}
    Let $t,d\in \N^+$, $f:X\to X$ be an acyclic Borel function and $E$ be a Borel equivalence relation on $X$ such that each $E$-class has finite diameter and $B_{2t(d+1)}(x)$ meets at most $(d+1)$ $E$-classes for every $x\in X$.
    Then there is a Borel $t$-forward-independent hitting set for $f$.
\end{prop}
\begin{proof}
    Define
    \begin{align*}
        A = \{x \in X : (x, f^k(x)) \notin E \text{ for all $k \in \N^+$} \}.
    \end{align*}
    Then $A$ is a Borel hitting set for $f$.
    Indeed, we have that $A = \bigcap_{k \in \N^+} A_k$, where
    $$A_k = \text{proj}_0(\text{graph}(f^k) \setminus E)$$
    is Borel for each $k \in \N^+$ by the Lusin--Suslin theorem.
    Now let $x \in X$.
    Then since $[x]_E$ has finite diameter, there is $k \in \N$ such that $(x, f^k(x)) \in E$ and $(x, f^{k + j}(x)) \notin E$ for every $j \in \N^+$.
    Thus $f^k(x) \in A$.

    Now set
    $$H = \{x \in A : f^j(x) \notin A \text{ for all $1 \leq j \leq t$} \}.$$
    Then $H = A \setminus (\bigcup_{1 \leq j \leq t} f^{-j}(A))$ is Borel and $t$-forward-independent.
    We claim that $H$ is hitting.
    Assume for contradiction that there is $x_0 \in X$ such that $f^k(x_0) \notin H$ for any $k \in \N$.
    Since $A$ is hitting, there are $x \in A$ and $\ell \in \N^+$ such that $f^{\ell}(x_0) = x$.
    However, since $x \notin H$, there is $1 \leq j_1 \leq t$ such that $f^j(x) \in A$. Let $j_0 = 0$, and then recursively construct a sequence $(j_n)_{n \geq 2}$ such that, for all $n \in \N$, $1 \leq j_n \leq t$ and $f^{j_1 + \cdots + j_n}(x) \in A$. Note that, for any $m \neq m'$, 
    $$(f^{j_0 + \cdots + j_m}(x), f^{j_0 + \cdots + j_{m'}}(x)) \notin E.$$
    So $B_{2t(d + 1)}(x)$ meets at least $d + 2$ classes of $E$, a contradiction.
\end{proof}

    \begin{remark}
        Note that a \emph{uniform} bound on diameter is not needed to construct the hitting sets in these proofs; boundedness alone is sufficient. Therefore, one can show that, if $\mcal{G}_f$ has finite \emph{Borel asymptotic separation index}, denoted $\asi_B(\mcal{G}_f)$ (see \cite[Definition~3.2]{cjmst2023}), then there is a Borel $r$-forward-independent hitting set for $f$ for each $r \in \N^+$. In particular, for a graph of the form $\mcal{G}_f$, if $\asi_B(\mcal{G}_f) < \infty$, then $\asi_B(\mcal{G}_f) \leq 1$. It is open whether $\asi_B(\mcal{G}) < \infty$ implies $\asi_B(\mcal{G}) \leq 1$ for general Borel graphs $\mcal{G}$ (see \cite{cjmst2023, qw2022,bw2023}).
    \end{remark}

\subsection{From hitting sets to finite asymptotic dimension}
\label{subsec:FIHAsdim}
    We start with a technical definition of a Borel map $c_H:X\to \{0,1\}$ that will be used to construct the witnesses to finite Borel asymptotic dimension.

    \medskip

    Let $t\in \N^+$ and set $s=6t$ and $r=4s^2$.
    Fix a decomposition
    \begin{equation}\label{eq:Decomp}
        \left\{0,1,\dots,\frac{r}{2}-1\right\}=I_0\sqcup \dots\sqcup I_{2m},
    \end{equation}
    where $m\in \N$ and $I_k$ is an interval that contains either $s$ or $s+1$ many elements for every $0\le k\le 2m$.
    This is possible since $\frac{r}{2}=2s^2=s(s+1)+(s-1)s$.
    
    Now let $f$ be a Borel function on a standard Borel space $X$, and let $H$ be a Borel $r$-forward-independent hitting set for $f$.
    For every $z\in H$, set $T_z=\bigcup_{j=1}^{\frac{r}{2}-1} f^{-j}(z)$, and for every $x\in X$, define $k(x)\in \N^+$ to be minimal such that $f^{k(x)}(x)\in H$.
    Observe that $T_z\cap T_{z'}=\emptyset$ for every $z, z'\in H$ with $z \neq z'$ as $H$ is $r$-forward-independent for $f$.

    We define a function $c_H:X\to \{0,1\}$ as follows.
    For every $x\in X$ such that $k(x)\ge \frac{r}{2}$, define 
    $$c_H(x)=\lfloor k(x)/s \rfloor \mod 2.$$
    Observe that if $z\in H$, then $k(z)\ge r$, so that $c_{H}(z)$ is defined; also note that, for any $x$, if $k(x)=\frac{r}{2}$, then $c_H(x)=0$.
    For any $z \in H$, if $c_H(z) = 0$, we set 
    $$c_H(x)=\lfloor k(x)/s \rfloor \mod 2$$
    for every $x\in T_z$.
    If $c_H(z) = 1$, we define 
    $$c_H(x)=k+1 \mod 2$$
    for every $x\in T_z$, where $0 \leq k \leq 2m$ is the unique integer such that $k(x)\in I_k$.
    Note that $c_H:X\to \{0,1\}$ is Borel and that $\{x : c_H(x) \neq c_H(f(x)) \}$ is a Borel $(s - 1)$-forward-independent set.

    \begin{prop}\label{pr:FIHtoCov}
        Let $t\in \N^+$, $f:X\to X$ be an acyclic Borel function and $H$ be a Borel $4(6t)^2$-forward independent hitting set.
        Define $c^{-1}_H(\{i\})=U_i$ for each $i\in \{0,1\}$, where $c_H$ is defined as in the previous paragraph.
        Then $U_0$ and $U_1$ are Borel sets such that $X=U_0\sqcup U_1$ and for each $i \in \{0, 1\}$, every equivalence class of $\mcal{F}_t(U_i)$ has diameter bounded by $28t+7$.
    \end{prop}
    \begin{proof}
        Recall from the definition of $c_H$ that $s=6t$.
        Define $\ell(x)\in \N^+$ to be minimal such that $c_H(x)\not=c_H(f^{\ell(x)}(x))$ and set
        $$e(x)=f^{\frac{s}{3}+\ell(x)}(x)$$
        for every $x\in X$.
        Note that we have $\ell(x)\le 2s+2$ for every $x\in X$.
        Indeed, if $k(x) > s + 1$, then by the definition of $c_{H}$, we have $\ell(x) \leq s + 1$.
        If $k(x) \leq s + 1$, then we have that $k(y) \ge r \geq s + 1$, where $y=f^{k(x)}(x)\in H$.        
        Hence, $\ell(f^{k(x)}(x)) \leq s + 1$ and consequently $\ell(x) \leq 2s + 2$.
        
        \begin{claim}\label{cl:BasicThree}
            For every $x\in U_i$, where $i\in \{0,1\}$, we have that
            $$\bigcup_{j=0}^{\frac{s}{3}} f^{-j}(e(x))\subseteq U_{1-i}.$$
        \end{claim}
        \begin{proof}
            As $e(x)=e(f^{\ell(x)-1}(x))$, we may assume without loss of generality that $\ell(x)=1$.
            First, we claim that $f^j(x)\in U_{1-i}\setminus H$ for every $1\le j\le \frac{s}{3}+1$.

            Indeed, if $f^j(x)\in H$ for any $1\le j\le \frac{s}{3}+1$, then, by the definition of $c_H$, we have $x\in U_{1-i}$, a contradiction.
            Similarly, assume for contradiction that $1\le j\le \frac{s}{3}+1$ is minimal such that $f^j(x)\not\in U_{1-i}$.
            Then $j\ge 2$, $f^{j-1}(x)\in U_{1-i}$ and we have that $x\in U_{1-i}$ as either $\lfloor k(x)/s \rfloor=\lfloor k(f^{j-1}(x))/s \rfloor \mod 2$ or $k(x),k(f^{j-1}(x))\in I_k$ for some $0\le k\le 2\ell$, where $I_k$ is from \eqref{eq:Decomp}.
            
            \medskip

            We finish the proof by showing that for every $0\le j\le \frac{s}{3}$ and $z\in f^{-j}(e(x))$, we have that $z\in U_{1-i}$.
            This clearly holds when $k(z)=k(f^{\frac{s}{3}+1-j}(x))$ by the previous paragraph.
            In particular, as $H$ is $r$-forward independent, it holds whenever $z\in H$.
            Now, the claim follows as, by the definition of $c_H$, we have that $c_H(w)=c_{H}(z)$ whenever $f^m(w)=z\in H$ and $0\le m\le s-1$.           
        \end{proof}
        
        \begin{claim}\label{cl:BasicTwo}
            Let $x,y\in X$ be such that $y\in [x]_{\mcal{F}_t(U_i)}$ for some $i\in\{0,1\}$.
            Then there is $k\in \N$ such that $f^k(y)=e(x)$.
        \end{claim}
        \begin{proof}
            Assume for contradiction that $f^k(y)\not= e(x)$ for any $k\in \N$.
            Then, by the definition of $\mcal{F}_t(U_i)$, there are $z,w\in [x]_{\mcal{F}_t(U_i)}$ such that $\rho_f(z,w)\le t$, $f^m(z)=e(x)$ for some $m\in \mathbb{N}$, and $f^k(w)\not= e(x)$ for any $k\in \N$.
            This is a contradiction, since the shortest path that connects $z$ and $w$ in $\mathcal{G}_f$ must contain $e(x)$ and go through
            $$\bigcup_{i=0}^{\frac{s}{3}} f^{-i}(e(x)),$$
            which is a subset of $U_{1 - i}$ by Claim~\ref{cl:BasicThree}.
            Hence, we have $\rho_f(w,z)\ge \frac{s}{3}>t$, a contradiction.         
        \end{proof}

        Observe that it follows from Claim~\ref{cl:BasicTwo} together with the assumption that $f$ is acyclic that, if $(x, y) \in \mcal{F}_r(U_i)$, then
        $$\rho_f(x,y)\le \rho_f(x,e(x))+\rho_f(y,e(y))\le \frac{2s}{3}+\ell(x)+\ell(y).$$
        This finishes the proof as we have that $\ell(x)\le 2s+2$ holds for every $x\in X$.
    \end{proof}

    \begin{prop}\label{pr:HittingToEqRel}
        Let $t\in \N^+$, $f:X\to X$ be an acyclic Borel function and $H$ be a Borel $4(6t)^2$-forward independent hitting set.
        Then there is a Borel equivalence relation $E$ on $X$ such that the $\rho_f$-diameter of every $E$-class is bounded by $28t+7$ and $B_t(x)$ meets at most $2$ $E$-classes for any $x \in X$.
    \end{prop}

    \begin{proof}
        We continue using the notation developed in the definition of $c_H$ and the proof of Proposition~\ref{pr:FIHtoCov}.
        For each $x \in X$, set
        $$L(x) = f^{\ell(x)}(x).$$
        Define a relation $E$ on $X$ by $(x,y)\in E$ if and only if $L(f^t(x)) = L(f^t(y))$.
        Then $E$ is Borel, since
        $$E = (L \circ f^t, L \circ f^t)^{-1}(\text{Diag}_{X\times X}).$$
        We show that $E$ works as required.
        
        To see that $E$ is uniformly $\rho_f$-bounded, let $x \in X$. Then
        $$[x]_E \subseteq \bigcup_{j \leq t + (2s + 2)} f^{-j}(L(f^t(x))).$$
        Indeed, if $y \in [x]_E$, then $L(f^t(y)) = L(f^t(x))$. So the distance from $y$ to $L(f^t(x))$ is at most $t + (2s + 2)$. Therefore, the diameter of $[x]_E$ is at most $2(t + (2s + 2))$.

        We finish the proof by showing that $B_t(x)$ meets at most $2$ $E$-classes for each $x \in X$.
        Note that $f^t[B_t(x)] = \{x, f(x), \dots, f^{2t}(x) \}$ and that
        $$\{x, f(x) \dots, f^{2t}(x) \} \subsetneq \{x, f(x), \dots, e(x) \}$$
        since $s=6t$.
        By the definition of $\ell(x)$ and Claim~\ref{cl:BasicThree}, there is $i\in \{0,1\}$ such that
        $$\{x, f(x), \dots, f^{\ell(x) - 1}(x) \}\subseteq U_i\text{ and }\{f^{\ell(x)}(x), \dots, e(x) \}\subseteq U_{1-i}.$$
        For each $y \in B_t(x)$, write $0\le j_y\le \frac{s}{3} + \ell(x)$ such that $f^t(y)=f^{j_y}(x)$.
        Let $y,y'\in B_t(x)$.
        It follows that if $j_y,j_{y'}< \ell(x)$, then $L(f^t(y))=L(x)=L(f^t(y'))$, and so $(y,y')\in E$.
        If $j_y,j_{y'}\ge \ell(x)$, then it follows that $L(f^t(y))=L(L(x))=L(f^t(y'))$, and so again $(y,y')\in E$.
    \end{proof}

    \subsection{Proof of Theorem~\ref{thm:EquivalentCharFinAsDim}}
    \label{subsec:AsdimEquiv}

    We show the following, which clearly implies Theorem~\ref{thm:EquivalentCharFinAsDim}.
    
    \begin{thm}\label{thm:Equivalences}
        Let $X$ be a standard Borel space, and let $f : X \rightarrow X$ be an acyclic Borel function. Then the following are equivalent:
        \begin{enumerate}
            \item [(i)] $\asdim^{\text{cov}}_B(\mcal{G}_f)$ is finite.
            \item [(i')] $\asdim^{\text{eq}}_B(\mcal{G}_f)$ is finite.
            \item [(ii)] $\asdim^{\text{cov}}_B(\mcal{G}_f)= 1$.
            \item [(ii')] $\asdim^{\text{eq}}_B(\mcal{G}_f)= 1$.
            \item [(iii)] For each $r > 0$, there is a Borel $r$-forward-independent hitting set for $f$.
        \end{enumerate}
    \end{thm}
    \begin{proof}
    (ii) $\implies$ (i) and (ii') $\implies$ (i') are trivial.

    \medskip
    
    (i) $\implies$ (iii). Suppose that $\asdim^{\text{cov}}_B(\mathcal{G}_f)=d\in \N$.
    Let $r\in \N^+$ and $U_0,\dots,U_d$ be as in Definition~\ref{def:Asdim}.
    For each $i\le d$, define
    $$A_i=\{x\in X:\forall k\in \N \exists n\in \N \ f^{k+n}(x)\in U_i\}=\bigcap_{k\in \N}\bigcup_{n=0}^\infty f^{-(k+n)}(U_i).$$
    It follows that $A_i$ is a Borel $f$-invariant set and $X=\bigcup_{i=0}^d A_i$.
    By Proposition~\ref{pr:CovToFIH}, applied to $X=A_i$ and $A=U_i$, we get that $f\upharpoonright A_i$ admits a Borel $r$-forward-independent hitting set for every $i\le d$, which gives (iii).
    
    \medskip

    (i') $\implies$ (iii). Suppose that $\asdim^{\text{eq}}_B(\mcal{G}_f) = d \in \N$.
    Let $r \in \N^+$ and set $r' = 2r(d + 1)$.
    By Definition~\ref{def:Asdim}, there is a uniformly $\rho_f$-bounded Borel equivalence relation $E$ on $X$ such that, for each $x\in X$, $B_{r'}(x)$ meets at most $(d+1)$ $E$-classes.
    Then by Proposition~\ref{pr:EqRelToHit}, applied with $t=r$, there is a Borel $r$-forward-independent hitting set for $f$, which gives (iii).

    \medskip

    (iii) $\implies$ (ii).
    Given $t\in \N^+$, apply Proposition~\ref{pr:FIHtoCov} to obtain Borel sets $U_0, U_1$ that cover $X$ and have the property that $\mcal{F}_r(U_i)$ is uniformly $\rho_f$-bounded for each $i \in \{0, 1\}$.

    \medskip

    (iii) $\implies$ (ii').
    Given $t \in \N^+$, apply Proposition~\ref{pr:HittingToEqRel} to obtain a uniformly $\rho_f$-bounded Borel equivalence relation $E$ on $X$ such that $B_t(x)$ meets at most $2$ $E$-classes for any $x \in X$.
    \end{proof}

    \begin{remark}
    \label{rem:new condition}
    One of the anonymous referees pointed out that there is another interesting equivalent condition that can be added to Theorem~\ref{thm:Equivalences}, which simplifies the proof of \emph{(iii) $\implies$ (ii')}:
    \begin{enumerate}
        \item [(iv)] For each $r> 0$, there is a Borel map $g : X \to X$ such that $\sup_{x \in X} \rho_f(x, g(x)) < \infty$ and $\vert g(B_r(x)) \vert \leq 2$ for all $x \in X$.
    \end{enumerate}
    We include a proof of the equivalence for completeness.

    \begin{proof}
    (iii) $\implies$ (iv).
    Let $r\in \mathbb{N}$.
    We put $s=2r$ and assume without loss of generality that $s=6t$ for some $t>0$.
    Define $r'=4s^2$.
    By (iii), there is a Borel $r'$-forward-independent hitting set $H$.
    Define $c_H$ as in Subsection~\ref{subsec:FIHAsdim}, and let $S = \{x \in X : c_H(x) \neq c_H(f(x)) \}$.
    As noted before Proposition~\ref{pr:FIHtoCov}, $S$ is a Borel $s$-forward-independent set.
    For each $x\in X$, define $j(x)\in \mathbb{N}$ to be minimal such that $f^{j(x)}(x)\in S$.
    The argument in the first paragraph of the proof of Proposition~\ref{pr:FIHtoCov} shows that $0\le j(x)\le 2s+1$ for every $x \in X$.
    Set $h(x)=f^{j(x)}(x)$, and observe that $h:X\to X$ is a Borel function.
    Define $g(x) = h(f^s(x))$.
    Clearly $g$ is Borel and $\sup_{x \in X} \rho_f(x, g(x)) < \infty$ as we have $\rho_f(x, g(x)) \leq s + (2s + 1) = 3s + 1$ for every $x\in X$.
    Given $x\in X$, note that $g(B_r(x)) = h(\{x, f(x), \dots, f^{s/2}(x) \})$, which has cardinality at most $2$ since $S$ is $(s - 1)$-forward-independent and $s - 1 > s/2$.
    Then $g$ is as desired.
    
    \medskip

    (iv) $\implies$ (ii').
    Let $r\in \mathbb{N}$ and take $g$ as in (iv) with parameter $r\in \mathbb{N}$.
    Define an equivalence relation $E$ on $X$ by $(x,y)\in E$ if and only if $g(x) = g(y)$.
    As $E=(g\times g)^{-1}(\text{Diag}_{X\times X})$, we have that $E$ is Borel.
    To see that $E$ is $\rho_f$-uniformly bounded, observe that, for every $x \in X$, the diameter of $[x]_E$ is at most $2M$, where $M$ is such that $\sup_{x \in X} \rho_f(x, g(x)) = M$.
    Finally, since $g(B_r(x))$ has cardinality at most $2$ for every $x\in X$, we conclude that $B_r(x)$ meets at most $2$ $E$-classes, and thus $\asdim^{\text{eq}}_B(\mcal{G}_f)= 1$.   
    \end{proof}
    \end{remark}

    \begin{remark}
        When $f$ is \emph{not} acyclic, define $C$ to be the collection of \emph{cyclic points} of $f$, i.e.,
        $$C=\{x\in X:\exists k \in \N, \, \ell \in \N^+ \ f^{k+\ell}(x)=f^k(x)\}.$$
        It is easy to show that $C$ is Borel and $f$-closed and that the connectedness relation of $\mcal{G}_f \rest C$ admits a Borel transversal $T$ that is hitting for $f \rest C$.
        From $T$, one can construct a witness to $\asdim_B(\mcal{G}_f \rest C) \leq 1$.
        Note, however, that for any $r \in \N^+$, if there is an $r$-forward-independent set for $f$, then $r \leq \ell$, where $\ell = 0$ if $f$ has a fixed point and $\ell$ is the minimal length of a cycle in $\mcal{G}_f$ if $f$ has no fixed points.
    \end{remark}
    
    \subsection{Proof of Theorem~\ref{thm:main}}
    \label{subsec:MainThm}
    By Theorem~\ref{thm:EquivalentCharFinAsDim}, it is enough to show that the set of finite-to-one acyclic Borel functions admitting a Borel $r$-forward-independent hitting set for every $r\in \N^+$ is $\mbf{\Sigma}^1_2$-complete.
    This latter claim follows directly from Theorem~\ref{thm:ComplexityFIH}.

\section{Open problems}
\label{sec:OpenProbs}

\subsection{Complexity of finite Borel asymptotic dimension for bounded-degree graphs}
It is an interesting problem whether Theorem~\ref{thm:main} remains true for bounded-degree (rather than locally finite) graphs.

    \begin{quest}
        Is the set of bounded-degree Borel graphs having finite Borel asymptotic dimension $\mbf{\Sigma}^1_2$-complete?
    \end{quest}

\subsection{Hypersmoothness of graphs generated by Borel functions} The graph generated by a Borel function on a standard Borel space is hypersmooth (see \cite[Theorem~8.1]{djk1994}). The following is a well-known open problem.

    \begin{quest}
    \label{quest:CommFuncs}
        Let $X$ be a standard Borel space, $d \in \N$, and $(f_i)_{i \leq d}$ be Borel functions on $X$ such that $f_i$ and $f_j$ commute for each $i, j \leq d$.
        Is the graph $\mcal{G}_{(f_i)_{i \leq d}}$ generated by $(f_i)_{i \leq d}$ hypersmooth?
    \end{quest}

The results of Section~\ref{sec:FinAsDim} indicate that Borel asymptotic dimension may be a useful tool for studying the above question in the case when $f_i$ is bounded-to-one for each $i \leq d$.
Furthermore, note that any graph of the form $\mcal{G}_{(f_i)_{i \leq d}}$ as in Question~\ref{quest:CommFuncs} can be identified with the Schreier graph of a Borel action of the semigroup $(\N^{d + 1}, +)$ on $X$.
An interesting research direction would be to systematically study the properties of Schreier graphs of Borel actions of other semigroups from the perspective of descriptive set theory.

\subsection{A characterization of the complexity of $\operatorname{CSP}_B(H)$} Let $H$ be a sinkless and sourceless digraph.
In Remark~\ref{rem:FunCSPvsCSP}, we demonstrate that if $\operatorname{CSP}_B^{\operatorname{function}}(H^p)$ is $\mbf{\Sigma}^1_2$-complete for some abstract walk $p$, then $\operatorname{CSP}_B(H)$ is also $\mbf{\Sigma}^1_2$-complete.
Although there is already a complete algebraic understanding of when $\operatorname{CSP}_B(H)$ is $\mbf{\Sigma}^1_2$-complete \cite{bkn2009,thornton2022}, a positive answer to the following problem, in combination with the proof of Theorem~\ref{thm:Complexity}, would give a new perspective on this result.

    \begin{quest}
    \label{q:LB}
        Let $H$ be a finite sinkless and sourceless digraph such that $\operatorname{CSP}_B(H)$ is $\mbf{\Sigma}^1_2$-complete.
        Is there an abstract walk $p$ for which $\operatorname{CSP}^{\operatorname{function}}_B(H^p)$ is $\mbf{\Sigma}^1_2$-complete?
    \end{quest}

    \begin{remark}
        A negative solution to Question~\ref{q:LB} was communicated to us by Libor Barto.
Let $H$ be defined as $V(H)=\{a,b,x,y\}$ and 
$$E(H)=\{(a,b),(b,a),(x,y),(y,x),(a,x),(a,y),(b,x),(b,y)\}.$$ 
In the notation of \cite{bkn2009}, this graph is called a $2$-tambourine; see \cite[Figure 1.]{bkn2009}.
It can be easily checked that $H$ does not retract onto a disjoint union of circles. Hence by the combination of \cite[Theorems~3.1, 3.2 and 3.3]{bkn2009} and \cite[Corollary~6.3]{thornton2022} we conclude that $\operatorname{CSP}_B(H)$ is $\mbf{\Sigma}^1_2$-complete.
On the other hand, it follows from Theorem~\ref{thm:Complexity}~(iii) that $\operatorname{CSP}^{\operatorname{function}}_B(H)$ is $\mbf{\Pi}^1_1$, and we claim that the same holds for $H^p$ for any abstract walk $p$.
Indeed, if the difference of the number of forward and backward arrows modulo $2$ in $H^p$ is $0$, then $H^p$ contains a loop (as $H$ restricted to $\{a,b\}$ is a $2$-cycle).
In particular, $\operatorname{CSP}^{\operatorname{function}}_B(H^p)$ is $\mbf{\Pi}^1_1$ by Theorem~\ref{thm:Complexity}~(i).
Otherwise, it can be easily checked that $H^p=H$.
    \end{remark}

\end{section}

\section*{Acknowledgments}
We would like to thank Zoltán Vidnyánszky, Alex Kastner, Andrew Marks, Sumun Iyer, and Forte Shinko for helpful discussions. In particular, we thank Alex Kastner for providing his note on \cite{tv2021} and a proof of Remark~\ref{rem:Alex} and Andrew Marks for helping to simplify the proof of Theorem~\ref{thm:ComplexityFIH}.
We thank the anonymous referees for helpful comments and suggestions. In particular, we thank one of the anonymous referees for pointing out an additional equivalent condition in Theorem~\ref{thm:Equivalences}, and for allowing us to include the suggested proof as Remark~\ref{rem:new condition}.

JG was supported by MSCA Postdoctoral Fellowships 2022 HORIZON-MSCA-2022-PF01-01 project BORCA grant agreement number 101105722, and by the Alexander von Humboldt Foundation in the framework of the Alexander von Humboldt Professorship of Daniel Kráľ endowed by the Federal Ministry of Education and Research.

\bibliographystyle{amsalpha}
\bibliography{references}

\end{document}